\documentclass[11pt]{amsart}
\usepackage{amsmath,amssymb,amsthm,amscd,verbatim}
\usepackage{graphicx}
\usepackage{epsfig}
\usepackage{hyperref}

\newtheorem{thm}{Theorem}

\newtheorem{cor}[thm]{Corollary}
\newtheorem*{thm*}{Theorem}

\theoremstyle{definition}

\theoremstyle{remark}

\newtheorem{qn}[thm]{Question}

\newtheorem*{acknowledgement}{Acknowledgments}

 \newcommand{\cP}{\mathcal{P}}

\newcommand{\bo}{\ensuremath{\mathrm{box}}}
\newcommand{\wco}{\ensuremath{\mathrm{wcol}}}
\newcommand{\co}{\ensuremath{\mathrm{col}}}
\newcommand{\di}{\ensuremath{\mathrm{dim}}}

\renewcommand{\le}{\leqslant}
\renewcommand{\leq}{\leqslant}
\renewcommand{\ge}{\geqslant}

\def\longequation{$$\vcenter\bgroup\advance\hsize by -9em%
\noindent\ignorespaces\refstepcounter{equation}}%
\makeatletter%
\def\endlongequation{\egroup\eqno(\theequation)$$\global\@ignoretrue}
\makeatother

\begin{document}
\title{Boxicity, poset dimension, and excluded minors}
\author{Louis Esperet} \address{Laboratoire G-SCOP (CNRS,
  Grenoble-INP), Grenoble, France}
\email{louis.esperet@grenoble-inp.fr}

\author{Veit Wiechert} \address{Institut f\"ur Mathematik, Technische Universit\"at Berlin, Berlin, Germany}

\thanks{Louis Esperet is partially supported by ANR Project GATO (\textsc{anr-16-ce40-0009-01}), and LabEx PERSYVAL-Lab
  (\textsc{anr-11-labx-0025}).}

\date{}
\sloppy

\begin{abstract}
In this short note, we relate the boxicity of graphs (and the
dimension of posets) with their
generalized coloring parameters. In particular, together with known
estimates, our results imply that
any graph with no $K_t$-minor can be represented as the intersection of
$O(t^2\log t)$ interval graphs (improving the previous bound of $O(t^4)$), and as the intersection of $\tfrac{15}2 t^2$
circular-arc graphs.
\end{abstract}
\maketitle

\section{Introduction}

The \emph{intersection} $G_1 \cap \cdots \cap G_k$ of $k$ graphs
$G_1,\ldots ,G_k$ defined on the same vertex set $V$, is the graph
$(V,E_1 \cap \ldots \cap E_k)$, where $E_{i}$ ($1 \le i \le k$)
denotes the edge set of $G_{i}$.  The
\emph{boxicity} $\bo(G)$ of a graph $G$, introduced by
Roberts~\cite{Rob69}, is defined as the smallest $k$ such that $G$ is
the intersection of $k$ interval graphs. 


\smallskip

Scheinerman proved that outerplanar graphs have boxicity at
most two \cite{Sch84} and Thomassen proved that planar graphs
have boxicity at most three \cite{Tho86}. Outerplanar graphs have no
$K_4$-minor and planar graphs have no $K_5$-minor, so a natural
question is how these two results extend to graphs with no $K_t$-minor for
$t\ge 6$.

It was proved in~\cite{EJ13} that if a graph has acyclic chromatic number at most $k$,
then its boxicity is at most $k(k-1)$. Using the fact that
$K_t$-minor-free graphs have acyclic chromatic number at most $O(t^2)$~\cite{HOQRS17}, it implies that graphs
with no $K_t$-minor have boxicity $O(t^4)$.
On the other hand, it was noted in~\cite{Esp16} that a result of Adiga, Bhowmick and Chandran~\cite{ABC10} (deduced from a result
of Erd\H os, Kierstead and Trotter~\cite{EKT91}) implies the existence of
graphs with no $K_t$-minor and with boxicity $\Omega(t\sqrt{\log t})$.

\smallskip

In this note, we relate the boxicity of graphs with their generalized
coloring numbers (see the next section for a precise
definition). Using this connection together with earlier results, we
prove the following result.

\begin{thm}\label{thm:bomin}
  There is a constant $C>0$ such that every $K_t$-minor-free graph has boxicity at most $C t^2 \log t$.
\end{thm}

Our technique can be slightly refined (and the bound can be slightly
improved) if instead of considering boxicity we consider 
a variant, in which we seek to represent graphs
as the intersection of circular-arc graphs
(instead of interval graphs as in the definition of boxicity).

\begin{thm}\label{thm:cibomin}
If $G$ has no $K_t$-minor, then $G$ can be represented as the
intersection of at most $\tfrac{15}2t^2$ circular-arc
graphs.
\end{thm}

\medskip

The {\em
  dimension} of a poset $\cP$, denoted by $\di(\cP)$, is the minimum number of linear orders
whose intersection is exactly $\cP$. Adiga, Bhowmick and
Chandran~\cite{ABC10} discovered a nice connection
between the boxicity of graphs and the dimension of posets, which has
the following consequence: for any poset $\cP$ with comparability graph
$G_{\cP}$, $\di(\cP)\le 2 \bo(G_{\cP})$. In particular, our main result implies
the following.

\begin{thm}\label{thm:dibomin}
There is a constant $C>0$ such that if ${\cP}$ is a poset whose comparability graph $G_{\cP}$ has no
$K_t$-minor, then $\di(\cP)\le Ct^2 \log t$.
\end{thm}

It should be noted that while Theorem~\ref{thm:bomin} directly implies
Theorem~\ref{thm:dibomin} (using the result of~\cite{ABC10} mentioned
above), deducing Theorem~\ref{thm:bomin}  from
Theorem~\ref{thm:dibomin} does not look as straightforward. Note also that a direct proof
of Theorem~\ref{thm:dibomin} can be obtained along the same lines as
that of Theorem~\ref{thm:bomin} (see~\cite{Wie17}).

\smallskip

Our result is based on a connection between
the boxicity of graphs and their weak 2-coloring number (defined in
the next section). Thus our result can also be seen as a connection
between the dimension of posets and the weak 2-coloring number of
their comparability graphs. Interestingly, similar connections between
the dimension of posets and weak colorings of their \emph{cover graphs}
have recently been discovered. The cover graph of a poset $\cP$ can
be seen as a minimal spanning subgraph of the comparability graph
$G_\cP$ of
$\cP$, from which $G_\cP$ can be recovered via transitivity. In
particular, the cover graph of $\cP$ can be much sparser than the
comparability graph of $\cP$ (for a chain, the first is a path while the
second is a complete graph). However, for posets of height two, the
comparability graph and the cover graph coincide.

It was proved by Joret, Micek, Ossona de Mendez and
Wiechert~\cite{JMOW17} that if $\cP$ is a poset of height at most $h$,
and the cover graph of $\cP$ has weak $(3h-3)$-coloring number at most
$k$, then $\di(P)\le 4^k$. For posets $\cP$ of height $h=2$, this
implies that the dimension is at most $4^k$, where $k$ is the weak
3-coloring number of the comparability graph of $\cP$. This will be
significantly improved in Section~\ref{sec:wcol} (see Theorem~\ref{thm:bowco}).

\medskip

The {\em adjacency poset} of a graph $G=(V,E)$, introduced by Felsner
and Trotter~\cite{FT00}, is the poset $(W, \leq)$ with $W=V \cup V'$,
where $V'$ is a disjoint copy of $V$, and such that $u \leq v$ if and
only if $u=v$, or $u\in V$ and $v\in V'$ and $u,v$ correspond to
adjacent vertices of $G$.  It was proved in~\cite{EJ13} that
for any graph $G$, the dimension of the adjacency poset of $G$ is at
most $2\, \bo(G)+\chi(G)+4$, where $\chi(G)$ is the chromatic number
of $G$. Since graphs with no $K_t$ minor have chromatic number $O(t \sqrt{\log
t})$~\cite{Kos84, Tho84}, this implies that the dimension of the
adjacency poset of any graph with no $K_t$-minor is $O(t^2\log t)$.

\section{Weak coloring}\label{sec:wcol}

Let $G$ be a graph and let $\Pi(G)$ denote the set of linear orders on $V(G)$.
Fix some linear order $\pi\in\Pi(G)$ for the moment.
We write $x<_\pi y$ if $x$ is smaller than $y$ in $\pi$, and we write
$x\le _\pi y$ if $x=y$ or $x<_\pi y$. For a set $S$ of vertices, $x
\le_\pi S$ means that $x\le _\pi y$ for every vertex $y\in S$. When
$\pi$ is clear from the context, we omit the subscript $\pi$ and write $<$ and $\le$ instead of
$<_\pi $ and $\le _\pi $.

For an integer
$r\ge 0$, we say that a
vertex $u$ is weakly $r$-reachable from $v$ in $G$ if there is a path $P$ of
length (number of edges) at most $r$ between $u$ and $v$, such that $u\le _\pi P$. In
particular, $u$ is
weakly 2-reachable from $v$ if $u\le _\pi v$, and either $u=v$, or $u$ and
$v$ are adjacent, or $u$ and $v$ have a common neighbor $w$ with
$u<_\pi w$.

\smallskip

The \emph{weak $r$-coloring number} of a graph $G$, denoted by
$\wco_r(G)$, is the minimum (over all linear orders $\pi\in \Pi(G)$) of the maximum (over all vertices $v$ of $G$) of the
number of vertices that are weakly $r$-reachable from $v$ with respect
to $\pi$. For more background on weak coloring numbers, the reader is
referred to~\cite{NO12}.

In this section we will consider the following slight variant of weak
coloring: let $\wco^*_r(G)$ be the minimum $k$ such that for some
linear order $\pi\in \Pi(G)$, there exists a coloring of the vertices of $G$ such
that for any vertex $v$ of $G$, all the vertices distinct from $v$ that are weakly
$r$-reachable from $v$ have a color that is distinct from that of
$v$. Note that the greedy algorithm trivially shows that for any graph
$G$ and integer $r\ge 0$, $\wco^*_r(G)\le \wco_r(G)$.

\smallskip

There is an interesting connection between $\wco^*_2(G)$ and the
\emph{star-chromatic number} $\chi_s(G)$ of $G$, which is defined as
the minimum number of colors in a proper coloring of the vertices of $G$, such that
any 4-vertex path contains at least 3 distinct colors. It was observed
in~\cite{NO03} that $\chi_s(G)$ can be equivalently defined as the
minimum number of colors in a coloring of some orientation of $G$,
such that any two vertices are required to have distinct colors if
they are connected by an edge, a directed 2-edge
path, or a 2-edge-path where the two edges are directed toward the
ends of the paths. An anonymous referee observed that if we add the
constraint that the orientation of $G$ is acyclic, the corresponding
graph parameter is precisely $\wco^*_2(G)$.

\smallskip

It is known that there exist graphs $G$ of unbounded boxicity with
$\wco_1(G)\le 2$~\cite{ABC10}. We now prove that the boxicity is bounded
by a linear function of  the weak 2-coloring number.

\begin{thm}\label{thm:bowco}
For any graph $G$, $\bo(G)\le 2 \,\wco^*_2(G)$.
\end{thm}

\begin{proof}
Let $G$ be a graph on $n$ vertices and let $c:=\wco_2^*(G)$.
 By definition, there exist a linear order $\pi$ on $V(G)$ and a vertex coloring $\phi$ with colors from the set $\{1,\ldots,c\}$, such that whenever a vertex $u$ is weakly $2$-reachable from another vertex $v$ with respect to $\pi$, then $\phi(u)\neq \phi(v)$.
 
 We aim to show that $G$ is the intersection of $2c$ interval graphs $I_1,\ldots,I_{2c}$.
 We associate to each color $i\in[c]$ the two interval graphs $I_{i}$ and $I_{i+c}$.
 Fix color $i$ for the moment.
 We explicitly define the intervals representing the vertices of $V(G)$ in $I_{i}$ and $I_{i+c}$, respectively. 
 Consider the vertices $v_1,\ldots,v_\ell$ that received color $i$ by $\phi$.
 By relabelling the vertices if needed, we may assume that $v_1<\cdots <v_\ell$ holds in $\pi$.
 
 We start with $I_{i}$.
 Here, we map $v_j$ ($1\leq j\leq \ell$) to the point $\{j\}$; and for
 every vertex $u$ that is not colored with $i$, we consider two cases:
 if $u$ has no neighbor colored $i$ we map $u$ to the point $\{n\}$,
 and otherwise we consider the minimal $k$ ($1\leq k\leq \ell$) such that $u$ is adjacent to $v_k$, and then we map $u$ to the interval $[k,n]$.
 Notice that $I_{i}$ is a supergraph of $G$.
 
 We now proceed with $I_{i+c}$.
 Here, we reverse the order of the vertices with color $i$, that is,
 we map $v_j$ ($1\leq j\leq\ell$) to the point $\{\ell-j+1\}$; and for
 every vertex $u$ not colored with $i$, we again map $u$ to the point
 $\{n\}$ if $u$ has no neighbor colored $i$, and otherwise we consider the maximal $k'$ ($1\leq k'\leq \ell$) such that $u$ is adjacent to $v_{k'}$, and then we map $u$ to the interval $[\ell-k'+1,n]$.
 Notice that $I_{i+c}$ is also a supergraph of $G$.
 In Figure~\ref{fig:box-rep} the two interval graphs $I_{i}$ and
 $I_{i+c}$ are illustrated by their induced box representation in
 dimensions $i$ and $i+c$.

\begin{figure}[htb]
 \centering
 \includegraphics[height=8cm]{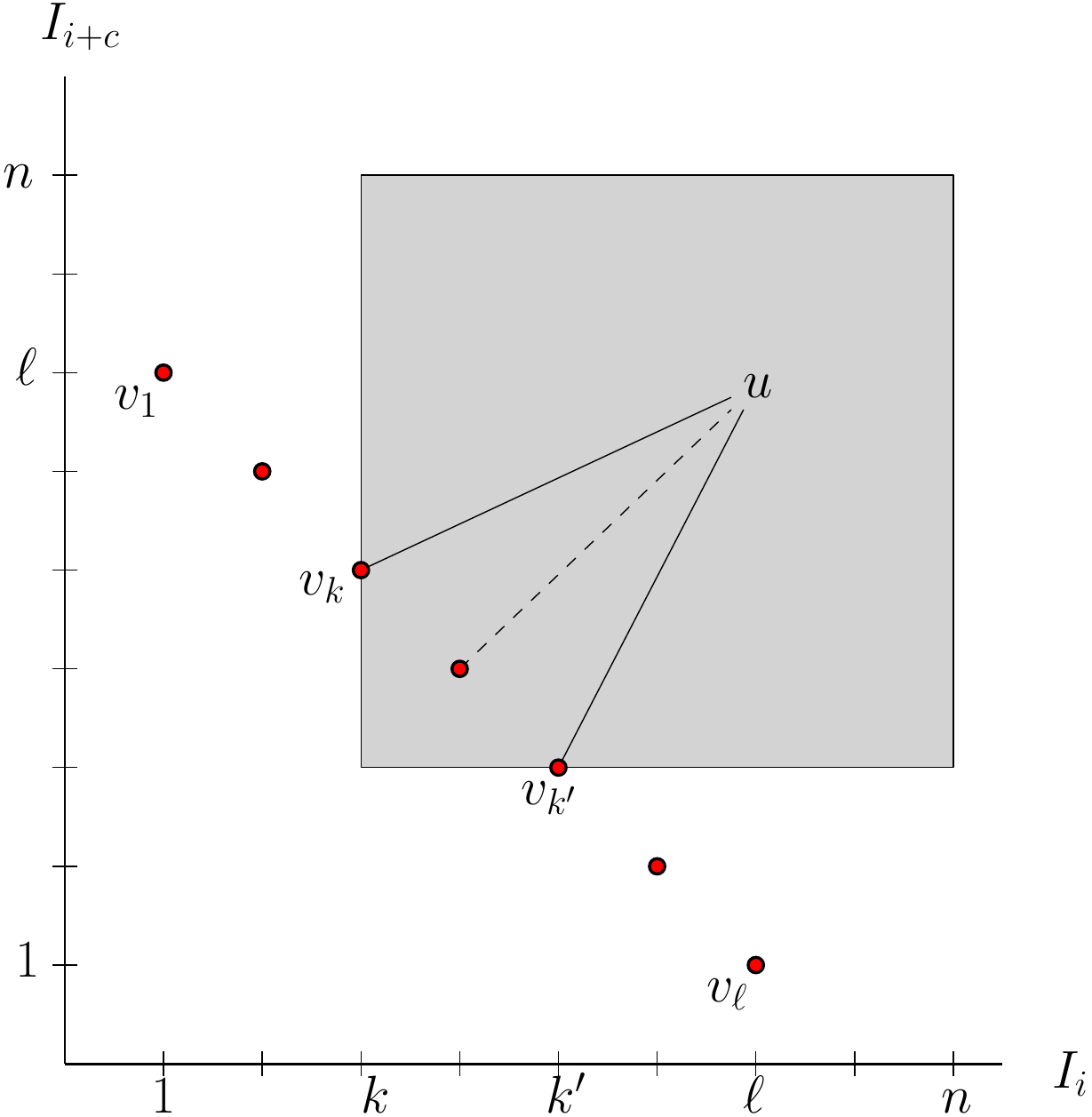}
 \caption{Illustration of $I_{i}$ and $I_{i+c}$ as the corresponding box representation. Vertices with color $i$ are mapped to the red points. Projections onto the two axis yield the intervals representing the vertices.}
 \label{fig:box-rep}
\end{figure}

 Next, we show that $G$ is the intersection of $I_1,\ldots,I_{2c}$.
 Since all involved interval graphs are supergraphs of $G$, we only
 need to show that for each pair of distinct non-adjacent vertices $u,v\in V(G)$ there is an interval graph $I_j$ ($1\leq j\leq 2c$) in which the two vertices are mapped to disjoint intervals.
 We may assume without loss of generality that $u<v$ in $\pi$.
 If $u$ and $v$ have the same color $i$, then their intervals are distinct points in $I_{i}$ (and also in $I_{i+c}$) and thus disjoint.
 So suppose that $u$ and $v$ have distinct colors $i$ and $j$, respectively.
 We assume for a contradiction that the intervals of $u$ and $v$ intersect in every interval graph $I_1,\ldots,I_{2c}$.
 This holds in particular in $I_{i}$ and $I_{i+c}$ (where $u$ is
 mapped to a point and $v$ to an interval containing point $\{n\}$);
 and from this we deduce that there are distinct vertices $x$ and $y$ with color $i$ such that $v$ is adjacent to both of them and $x<u<y$ in $\pi$.
 However, since we assumed that $u<v$ in $\pi$, this implies
 that $x<u<y<v$ or $x < u < v < y$. It follows that $x$ is weakly $2$-reachable from $y$ with respect to $\pi$, as is witnessed by the path $x,v,y$.
 This is a contradiction to the properties of the coloring $\phi$.
 
 We conclude that $G$ is indeed the intersection of $I_1,\ldots,I_{2c}$, and thus $\bo(G)\leq 2c$. 
\end{proof}

It is known that planar graphs have weak 2-coloring number (and thus
$\wco_2^*)$ at most 30~\cite{HOQRS17}, so this implies that their
boxicity is at most 60 (this is significantly worse than the result of
Thomassen~\cite{Tho86}, who proved that planar graphs have boxicity at
most 3). Given two integers $s$ and $t$, let $K_{s,t}^*$ denote the
complete join of $K_s$ and $\overline{K_t}$. Using recent bounds on weak 2-coloring numbers by Van den Heuvel
and Wood (Proposition 28 in~\cite{HW19}), Theorem~\ref{thm:bowco}
directly implies the following.

\begin{thm}
If $G$ does not contain $K_{s,t}^*$ as a minor, then $\bo(G)\le 5s^3(t-1)$.
\end{thm}

In particular, when $s$ is a constant, the boxicity is linear in
$t$. The result also directly implies that the boxicity of
$K_t$-minor-free graphs is $O(t^3)$. This is also the order of magnitude of
the best known bound on the weak 2-coloring number of $K_t$-minor-free
graphs. To improve the bound on the boxicity of  $K_t$-minor-free
graphs, we will now use  $\wco^*_2$
as an alternative to $\wco_2$. We believe that considering
$\wco^*_r$ instead of $\wco_r$ might yield to significant improvements
in other problems as well.

\smallskip

It is proven in~\cite{NO03} that if every minor of a graph $G$ has average degree at most $d$, then the star-chromatic number $\chi_s(G)$ is $O(d^2)$.
A closer look at the proof contained in the paper reveals that this
bound also holds for $\wco_2^*(G)$. We will indeed prove that a slightly stronger
statement holds. Given a graph $H$, a \emph{subdivision} of $H$ is a graph obtained
from $H$ by subdividing some of the edges of $H$ (i.e. replacing
them by paths). The subdivision is said to be an \emph{$(\le
\ell)$-subdivision} if each edge is subdivided at most $\ell$ times (i.e. replaced
by a path on at most $\ell+1$ edges). Given a half-integer $r\ge 0$, we denote by $\widetilde{\nabla}_r(G)$ the
maximum average degree of a graph $H$ such that $G$ contains an $(\le
2r)$-subdivision of $H$ as a subgraph. Given an integer $r\ge 0$, let
$\nabla_r(G)$ be the maximum average degree of a graph that can be
obtained from $G$ by contracting disjoint balls of radius at most $r$.
For more on these notions and their
connections with generalized coloring parameters, the
reader is referred to the monograph~\cite{NO12}.

We now prove the following (the first part of the result is a simple
rewriting of the original argument of~\cite{NO03}, while the second
part was suggested to us by Sebastian Siebertz).

\begin{thm}
  For any graph $G$, $\wco_2^*(G)\le 3 \nabla_0(G)^2+ 1+\min(\nabla_0(G) \nabla_1(G),\nabla_0(G)^2
  \widetilde{\nabla}_{1/2}(G))$.
\end{thm}

\begin{proof}
Any subgraph of $G$ has average
degree at most $k=\nabla_0(G)$ and in particular $\wco_1(G)\le k$ (i.e. $G$ is
$k$-degenerate). Let $\pi\in \Pi(G)$ be an order such that for any
vertex $u$, at most $k$ neighbors $v$ of $u$ are such that $v<_\pi
u$ (in the remainder of the proof, we write $<$ instead of $<_\pi$
whenever there is no risk of confusion). Let $H$ be obtained from $G$ by adding an edge between $u$ and $w$,
for each $u<v<w$ such that $uv$ and $vw$ are edges of $G$, and by
adding an edge  between $x$ and $y$ for each $x<y<z$ such that $xz$
and $yz$ are edges of $G$. Observe that $\wco_2^*(G)\le \chi(H)$,
where $\chi(H)$ denotes the chromatic number of $H$. Thus, it is sufficient to
prove that $H$ is $c$-colorable, with $c=3k^2+1+\min(k \nabla_1(G) ,k^2
\widetilde{\nabla}_{1/2}(G))$. In order to do so,
we will indeed prove that any subgraph of $H$ has average degree at
most $c-1$, which implies that $H$ is $(c-1)$-degenerate and thus
$c$-colorable. Consider a subset $A$ of vertices of $G$. Each edge $uv$ of $H[A]$, with $u<v$,  corresponds to (at least) one of
these cases:

\begin{itemize}
\item $uv$ is an edge of $G$ (there are at most $k|A|/2$ such edges,
  since all subgraphs of $G$ have average degree at most $k$).
  \item there is a vertex $x$ in $G$ (not necessarily in $A$) with
    $u<x<v$, such that $ux$ and $xv$ are edges of $G$ (there are at
    most $k^2|A|$ such edges, by definition of $\pi$).
  \item there is a vertex $w \in A$ with $u<v<w$ such that $uw$ and
    $vw$ are edges of $G$ (there are at most $|A|k(k-1)/2$ such edges,
    since for each $w$ in $A$ there are at most $k(k-1)/2$ pairs of
    neighbors of $w$ in $G$ preceding it in $\pi$).
    \item there is a vertex $w \notin A$ with $u<v<w$ such that $uw$ and
    $vw$ are edges of $G$ (in this case let us say that the edge $uv$
    of $H[A]$ is \emph{special}).
\end{itemize}

It follows from the observations above that there are at most $3k^2|A|/2$
non-special edges in $H[A]$.
We now bound the number of special edges $uv$ of $H[A]$ in two
different ways. For each
vertex $x\not\in A$, consider the (at most $k$) edges $yx$ of $G$ with
$y\in A$ and $y<x$, and label them with distinct integers from the set
$\{1,\ldots,k\}$. For each $1\le i \le k$, observe that the edges
labelled $i$ form disjoint unions of stars, centered in vertices of
$A$. For $1\le i \le k$, let $G_i$ be the graph obtained from $G$ by
contracting each of these stars labelled $i$ into a single
vertex. Note that each special edge of $H[A]$ corresponds to an edge
in at least one of the graphs $G_i[A]$. Since each $G_i[A]$ was
obtained from $G$ by contracting disjoint balls of radius at most 1,
each $G_i[A]$ contains at most $\nabla_1(G)|A|/2$ edges, and thus
$H[A]$ has at most $k\nabla_1(G)|A|/2$ special edges. It follows that
$H$ has average degree at most $3k^2+k\nabla_1(G)$, as desired.

For each pair $i,j$ with $1\le i<j\le k$, consider the graph $G_{ij}$
with vertex set $A$, such that any two vertices $u,v\in A$ are connected by an
edge if in $G$, $u$ and $v$ are connected by a path on 2 edges, one
labelled $i$ and the other labelled $j$. Observe that each special
edge of $H[A]$ is an edge of some $G_{ij}$, and $G$ contains a
1-subdivision of each $G_{ij}$. It follows that $H[A]$ has at most
$k^2  \widetilde{\nabla}_{1/2}(G)|A|/2$ special edges, and thus
average degree at most $3k^2+k^2 \widetilde{\nabla}_{1/2}(G)$, as desired.
  \end{proof}
  
Since graphs with no $K_t$-minor have average degree
$\mathcal{O}(t\sqrt{\log t})$~\cite{Kos84,Tho84}, it follows that
these graphs have $\wco^*_2(G)=O(t^2\log t)$. We thus obtain
Theorem~\ref{thm:bomin} as a direct consequence of
Theorem~\ref{thm:bowco}.

\medskip

A classic result~\cite{BT98,KS94} states that graphs with no subdivision of $K_t$
have average degree $O(t^2)$. So, for these graphs $\nabla_0$ and
$\widetilde{\nabla}_{1/2}$ are of order $O(t^2)$. An immediate consequence is the
following.

\begin{cor}\label{cor:bominsubdv}
There is a constant $C>0$ such that if $G$ has no subdivision of
$K_t$, then $\bo(G)\le Ct^6 $.
\end{cor}

\section{Strong coloring and circular-arc
  graphs}\label{sec:scol}\bigskip

The purpose of this section is to prove that if we consider a slightly
larger class of graphs (circular-arc graphs instead of interval
graphs), we can gain a multiplicative factor of $\log t$ in Theorem~\ref{thm:bomin}.

\medskip

A \emph{circular interval} is an interval of the unit circle, and a \emph{circular-arc graph} is the intersection graph of
a family of circular intervals. Equivalently, we can define a
\emph{circular interval} of $\mathbb{R}$ as being either an interval of
$\mathbb{R}$, or the (closed) complement of an interval of
$\mathbb{R}$. Note that this defines the same intersection graphs, and
we will use whatever formulation is the most convenient, depending on
the situation.

\medskip

The \emph{circular dimension} of a graph $G$, denoted by
$\di^\circ(G)$, is the minimum integer $k$ such that $G$ can be
represented as the intersection of $k$ circular-arc graphs. This
parameter was introduced by Feinberg~\cite{Fei79}. Since
every interval graph is a circular-arc graph, $\di^\circ(G)\le \bo(G)$ for any
graph $G$.

\medskip

Let $G$ be a graph, let $\pi\in \Pi(G)$, and let $r\ge 0$ be an integer.
Following~\cite{KY03}, we say that a
vertex $u$ is \emph{strongly $r$-reachable} from $v$ if there is a path $P$ of
length at most $r$ between $u$ and $v$, such that $u\le_\pi P$ and $v
\le_\pi P-u$. In
particular, $u$ is
strongly 2-reachable from $v$ if $u\le_\pi v$, and either $u=v$, or $u$ and
$v$ are adjacent, or $u$ and $v$ have a common neighbor $w$ with
 $u<_\pi v<_\pi w$.

The \emph{strong $r$-coloring number} of a graph $G$, introduced
in~\cite{KY03} and denoted by
$\co_r(G)$, is the minimum (over all linear orders $\pi \in \Pi(G)$) of the maximum (over all vertices $v$ of $G$) of the
number of vertices that are strongly $r$-reachable from $v$.

\begin{thm}\label{thm:cagco}
For any graph $G$,
$\di^\circ(G) \le 3\co_2(G)$.
\end{thm}

\begin{proof}
The proof proceeds similarly as the proof of
Theorem~\ref{thm:bowco}. Let $n$ be the number of vertices in $G$. We consider a total order $\pi\in \Pi(G)$ on the vertices
of $G$ 
such that for any $v$, at most $c=\co_2(G)$ vertices are strongly
2-reachable from $v$. Again, any notion of order between the vertices of
$G$ in this proof will implicitly refer to $\pi$. As before, we start by
greedily coloring $G$, with at most $c$ colors, such that for any $v$
and any vertex $u\ne v$ that is strongly 2-reachable from $v$, the
colors of $u$ and $v$ are distinct. For each color class $1\le i \le
c$, we consider the two interval graphs
$I_i$ and $I_{i+c}$ of the proof of Theorem~\ref{thm:bowco}, and a
circular-arc graph $I_{i+2c}$ defined as follows. Let
$v_1<\ldots <v_\ell$ be the vertices colored $i$ in $G$. Again, each
vertex $v_j$ ($1\le j\le \ell$) is mapped to the point $\{j\}$. Each vertex $v$ not colored $i$ is
mapped (1) to the point $\{n\}$ if $v$ has no neighbor colored $i$,
(2) to
the interval $[j,n]$ if $v_j$ is the unique neighbor of $v$
colored $i$, and otherwise (3) to the complement of the open interval
$(j,k)$, where $v_j$ and $v_k$ are the smallest and second smallest neighbors of $v$ colored
$i$ (with respect to $\pi$). An example of construction of $I_{i+2c}$ is
illustrated in Figure~\ref{fig:box-rep2}.

\begin{figure}[htb]
 \centering
 \includegraphics[height=3.5cm]{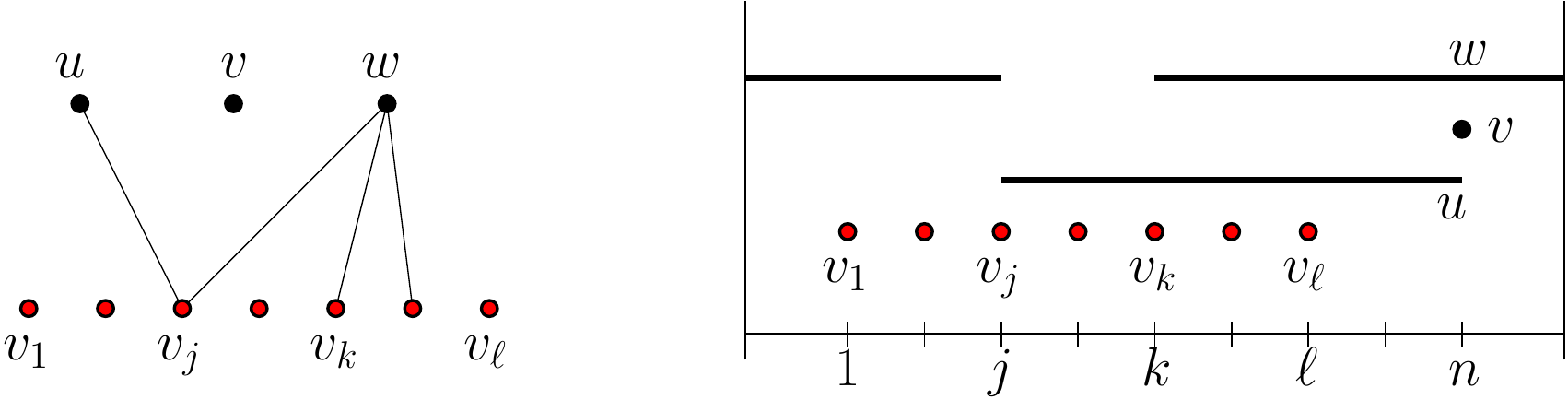}
 \caption{Left: A graph, with the vertices colored $i$ depicted in
   red; Right: The corresponding circular-arc graph $I_{i+2c}$.}
 \label{fig:box-rep2}
\end{figure}

We now prove that $G$ is precisely the intersection of the graphs
$I_i$ for $1\le i \le
3c$, which will show that $G$ is the intersection of at most $3c=3
\co_2(G)$ circular-arc graphs. We already proved in the previous
section that for each $1\le i \le 2c$,
the graphs $I_i$ are supergraphs of $G$. We prove that it is also
the case for the graphs $I_{i+2c}$ with $1 \le i \le c$. Observe
that in the
graph $I_{i+2c}$, any vertex $v$ not
colored $i$ is adjacent to all the vertices not colored $i$, and to
all its neighbors in $G$ that are colored $i$. Since there is no egde
between two vertices colored $i$ in
$G$, it follows that every edge $uv$ in $G$ is also an edge of
$I_{i+2c}$.

Hence, in order
to prove that $G$ is precisely the intersection of the graphs
$I_i$ for $1\le i \le
3c$, it is sufficient to prove that each non-edge $uv$
of $G$ is also a non-edge in a graph $I_i$ for some $1\le i \le
3c$. Consider two non-adjacent vertices $u<v$ in $G$.
We can assume that $u$ and $v$ have distinct colors $i$ and $j$, respectively
(otherwise $uv$ is a non-edge in the three graphs $I_i$, $I_{i+c}$,
and $I_{i+2c}$
corresponding to their common color class). If $v$ has no neighbor colored
$i$ in $G$, then $v$ has no neighbor colored $i$ in each of the three graphs $I_i$, $I_{i+c}$,
and $I_{i+2c}$, and thus $u$ and $v$ are non-adjacent in each of these
graphs. If $v$ has a unique neighbor colored $i$ in $G$,
call it $w$ (it is different from $u$, since $u$ and $v$ are
non-adjacent), then it follows from the construction of $I_i$ and
$I_{i+c}$ that $w$ is the unique neighbor of $v$ colored $i$ in
$I_i\cap I_{i+c}$, and thus $u$ and $v$ are non-adjacent in $I_i\cap
I_{i+c}$. So we can assume that $v$ has
at least two neighbors colored $i$. As in the proof of
Theorem~\ref{thm:bowco} (using the definition of $I_i$ and $I_{i+c}$) we can assume that $v$ has two neighbors $x$
and $y$ colored $i$, such that $x<u<y$. Take $x$ and $y$ minimal (with
respect to $\pi$) with this property. 

By the definition of the strong
2-coloring number, we can assume that at most one neighbor of $v$
colored $i$ precedes $v$ in $\pi$ (since otherwise the smaller neighbor would
be strongly 2-reachable from the larger neighbor, via $v$, which would contradict
the fact that the two neighbors have the same color). Hence, it
follows that $x$ and $y$ are respectively the smallest and second
smallest neighbors of $v$ colored $i$. But since
$x<u<y$ and $u$ is colored $i$, it follows from the definition of $I_{i+3c}$ that $u$ and $v$ are non-adjacent in $I_{i+3c}$,
as desired.
\end{proof}

The following result was recently proved by Van den Heuvel, Ossona de Mendez,
Quiroz,  Rabinovich and Siebertz~\cite{HOQRS17}.

\begin{thm}\label{thm:cot}
If $G$ has no $K_t$-minor, then $\co_2(G)\le \tfrac52 (t-1) (t-2)$.
\end{thm}

Together with Theorem~\ref{thm:cagco}, this immediately implies
Theorem~\ref{thm:cibomin}.

\medskip

The $\log t$ factor between Theorems~\ref{thm:bomin}
and~\ref{thm:cibomin} raises some interesting questions about the
parameter $\di^\circ$. It is known that every $n$-vertex graph has
boxicity at most $n/2$, and equality holds only for the complete graph
$K_n$ ($n$ even) minus a perfect matching. However this graph is a
circular-arc graph (see Figure~\ref{fig:stex}) and thus has circular dimension equal to 1.

\begin{figure}[htb]
 \centering
 \includegraphics{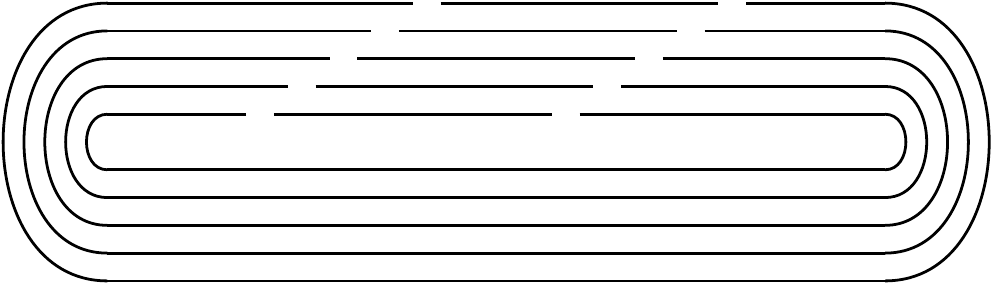}
 \caption{A circular-arc graph representation of the complete graph $K_{10}$ minus a perfect matching.}
 \label{fig:stex}
\end{figure}

\begin{qn}
What is the maximum circular dimension of a graph on $n$ vertices?
\end{qn}

It was observed in ~\cite{She80} that there are $2^{\Theta(b n
  \log n)}$ $n$-vertex
graphs of circular dimension at most $b$, and thus almost all
$n$-vertex graphs have circular dimension $\Omega(n/\log
n)$.

\smallskip

It is known that every graph of maximum degree $\Delta$ has boxicity
$O(\Delta \log^{1+o(1)} \Delta)$~\cite{SW18}, while there are graphs of maximum degree $\Delta$ with boxicity
$\Omega(\Delta \log \Delta)$~\cite{ABC10}.

\begin{qn}
What is the maximum circular dimension of a graph of maximum degree $\Delta$?
\end{qn}

Since there are $2^{\Theta(\Delta n
  \log n)}$ $n$-vertex
graphs of maximum degree $\Delta$, it follows that almost all graphs
of maximum degree $\Delta$ have circular dimension
$\Omega(\Delta)$. On the other hand, it was proved by Aravind and
Subramanian~\cite{AS09} that every graph of maximum degree $\Delta$
has circular dimension $O(\Delta \log \Delta/\log \log \Delta)$.

\begin{acknowledgement} This work was initiated during the Order \&
  Geometry Workshop organized at the Gu\l towy Palace near Poznan,
  Poland, on September 14--17, 2016. We thank the organizers and
  participants for the interesting discussions and nice atmosphere. We
  also thank Sebastian Siebertz for his remark on graphs excluding a
  subdivision of $K_t$, and  N.R. Aravind for
  calling references~\cite{AS09}, \cite{Fei79}, and~\cite{She80} to
  our attention. The
  presentation of the paper was significantly improved thanks to the
  comments and suggestions of two
  anonymous reviewers. 
\end{acknowledgement}


\begin{thebibliography}{99}



     

\bibitem{ABC10} A.~Adiga, D.~Bhowmick, and L.S.~Chandran,
     \emph{Boxicity and Poset Dimension}, SIAM J. Discrete Math. {\bf 25(4)} (2011), 1687--1698.


\bibitem{AS09} N.R. Aravind and C.R. Subramanian, \emph{Intersection
    Dimension and Maximum Degree}, Electronic Notes in Discrete
  Mathematics {\bf 35} (2009), 353--358.

\bibitem{BT98} B. Bollob\'as and A. Thomason, \emph{Proof of a
    conjecture of Mader, Erd\H{o}s and Hajnal on topological complete
    subgraphs}, European J. Combin. {\bf 19(8)} (1998), 883--887.

\bibitem{EKT91}  P. Erd\H os, H.A. Kierstead, and W.T. Trotter,
  \emph{The dimension of random ordered sets}, Random Structures Algorithms {\bf 2(3)} (1991), 253--275.

\bibitem{Esp16} L. Esperet, \emph{Boxicity and topological invariants}, 
European J. Combin. {\bf 51} (2016), 495--499.

\bibitem{EJ13} L. Esperet and G. Joret, \emph{Boxicity of graphs on surfaces},
  Graphs Combin. {\bf 29(3)} (2013), 417--427.

\bibitem{Fei79} R.B. Feinberg, \emph{The circular dimension of a
    graph}, Discrete Math. {\bf 25} (1979), 27--31.

  \bibitem{FLT10} S.~Felsner, C.M.~Li, and W.T.~Trotter, \emph{Adjacency
    posets of planar graphs}, Discrete Math. {\bf 310} (2010),
  1097--1104.

  \bibitem{FT00} S.~Felsner and W.T.~Trotter, \emph{Dimension, graph and
    hypergraph coloring}, Order {\bf 17} (2000), 167--177.

\bibitem{HOQRS17} J. van den Heuvel, P. Ossona de Mendez, D. Quiroz,
  R. Rabinovich and S. Siebertz, \emph{On the generalised
    colouring numbers of graphs that exclude a fixed minor},
  European J. Combin {\bf 66} (2017), 129--144.

\bibitem{HW19}  J. van den Heuvel and D.R. Wood,
\emph{Improper colourings inspired by Hadwiger's Conjecture},
J. Lond. Math. Soc. {\bf 98(1)} (2018), 129--148.

\bibitem{JMOW17} G. Joret, P. Micek, P. Ossona de Mendez, and V. Wiechert,
\emph{Nowhere Dense Graph Classes and Dimension},  Manuscript, \href{http://arxiv.org/pdf/1708.05424}{\tt arXiv:1708.05424}, 
2017.

\bibitem{KY03} H.A. Kierstead and D. Yang, \emph{Orderings on graphs
    and game coloring number}, Order {\bf 20} (2003), 255--264.

\bibitem{KS94} J. Koml\'os and E. Szemer\'edi, \emph{Topological
    cliques in graphs}, Combin. Prob. Comput. {\bf 3} (1994), 247--256.

\bibitem{Kos84} A.~Kostochka, \emph{Lower bound of the Hadwiger number
    of graphs by their average degree}, Combinatorica {\bf 4} (1984),
  307--316.

  \bibitem{MoTh} B.~Mohar and C.~Thomassen, \emph{Graphs on Surfaces},
  Johns Hopkins University Press, Baltimore, 2001.

\bibitem{NO03} J. Ne\v set\v ril and P. Ossona de Mendez,
  \emph{Colorings and homomorphisms of minor closed
  classes}, Discrete \& Computational Geometry, The Goodman-Pollack
  Festschrift, volume 25 of Algorithms and Combinatorics (2003),
  651--664.

\bibitem{NO12} J. Ne\v set\v ril and P. Ossona de Mendez,
  \emph{Sparsity -- Graphs, Structures, and
    Algorithms}, Springer-Verlag, Berlin, Heidelberg, 2012.

\bibitem{Rob69} F.S.~Roberts, \emph{On the boxicity and cubicity of a
  graph}, In: {\em Recent Progresses in Combinatorics}, Academic
  Press, New York, 1969, 301--310.



\bibitem{Sch84} E.R.~Scheinerman, \emph{Intersection classes and
  multiple intersection parameters}, Ph.D. Dissertation, Princeton
  University, 1984.

\bibitem{SW18} A. Scott and D.R. Wood, \emph{Better bounds for
    poset dimension and boxicity}, Manuscript, \href{http://arxiv.org/pdf/1804.03271}{\tt arXiv:1804.03271},
2018. 


\bibitem{She80} J.B. Shearer, \emph{A note on circular dimension},
  Discrete Math. {\bf 29(1)} (1980), 103.


\bibitem{Tho84} A.~Thomason, \emph{An extremal function for
    contractions of graphs}, Math. Proc. Cambridge Philos. Soc. {\bf
    95} (1984), 261--265.

\bibitem{Tho86} C.~Thomassen, \emph{Interval representations of planar
  graphs}, J.~Combin.\ Theory Ser.~B~{\bf 40} (1986), 9--20.

\bibitem{Wie17} V. Wiechert, \emph{Cover Graphs and Order Dimension}, PhD thesis,
  Technische Universit\"at Berlin, 2017.

\end{thebibliography}
\end{document}